\documentclass{birkjour}
\usepackage{amsmath,amssymb,amsthm}

\theoremstyle{definition}
\newtheorem{theorem}{Theorem} 

\newtheorem{corollary}{Corollary} 
\newtheorem{lemma}{Lemma} 

\newcommand{\set}[1]{\{#1\}}

\newcommand{\real}{\mathbb{R}}
\newcommand{\nn}{\mathbb{N}}
\newcommand{\sd}{\,|\,} 
\newcommand{\goesto}{\rightarrow} 
\DeclareMathOperator{\Exp}{E} 

\begin{document}

\title[Identities for partial Bell polynomials]
 {Identities for partial Bell polynomials derived from identities for weighted integer compositions\footnote{The final publication is available at Springer via http://dx.doi.org/10.1007/s00010-015-0338-2}}

\author[Eger]{Steffen Eger}

\address{%
Robert Mayer Str. 10\\
P.O. Box 154\\
D-60325 Frankfurt am Main\\
Germany}

\email{eger.steffen@gmail.com}





\begin{abstract}
  We discuss closed-form formulas for the $(n,k)$-th partial Bell
  polynomials derived in Cvijovi\'c \cite{cvijovic}. We show that
  partial Bell polynomials are special cases of weighted integer
  compositions, and demonstrate how the identities for partial Bell polynomials
  easily follow from more general identities for weighted integer
  compositions. We also provide short and elegant probabilistic proofs
  of the latter, in terms of sums of discrete integer-valued random
  variables. Finally, we outline further identities for the
  partial Bell polynomials.
\end{abstract}

\maketitle

\section{Introduction}
In a recent note, Cvijovi\'c \cite{cvijovic} has derived three new
identities --- which we list in Equations \eqref{eq:id1},
\eqref{eq:id2}, and \eqref{eq:id3} below --- 
for the $(n,k)$-th partial Bell polynomials in the variables
$x_1,x_2,\ldots,x_{n-k+1}$, which are 
defined as\footnote{Another way to define the partial Bell polynomials
  is by the formal power series expansion
  $\frac{1}{k!}\left(\sum_{m\ge 1}\frac{x_m}{m!}t^m\right)^k =
  \sum_{n\ge k} \frac{B_{n,k}(x_1,x_2,\ldots,x_{n-k+1})}{n!}t^n$. This
would also immediately lead to Corollary \ref{corr:1} below if we
defined the subsequent concept of weighted integer compositions by the
formal power series 
expansion $\left(\sum_{m\ge 0}f(m)t^m\right)^k = \sum_{n\ge
  0}\binom{k}{n}_ft^n$, with 
notation as explicated below.}
\begin{equation}\begin{split}
    \label{eq:bell}
  B_{n,k}(x_1,x_2,\ldots,x_{n-k+1}) &= \\
  \sum \frac{n!}{\ell_1!\cdots
    \ell_{n-k+1}!}&\left(\frac{x_1}{1!}\right)^{\ell_1}\cdots\left(\frac{x_{n-k+1}}{(n-k+1)!}\right)^{\ell_{n-k+1}},
  \end{split}
\end{equation}
where the summation is over all solutions in nonnegative integers
$\ell_1,\ldots,\ell_{n-k+1}$ 
of
$\ell_1+2\ell_2+\cdots+(n-k+1)\ell_{n-k+1}=n$ and
$\ell_1+\ell_2+\cdots+\ell_{n-k+1}=k$. Such identities may be
important for the efficient evaluation of the partial Bell
polynomials, as hinted at in Cvijovi\'c \cite{cvijovic}, who also indicates
applications of the polynomials. Importantly, they generalize the
Stirling numbers $S(n,k)$ of the second kind, since
$S(n,k)=B_{n,k}(1,\ldots,1)$. 

The purpose of the present note is to show that these identities are
special cases of identities for \emph{weighted integer compositions}
and to outline very short proofs for the more general identities,
based on sums of discrete random variables. 
Our main point is to compile (identical
or very similar) results
developed within heterogeneous communities of research, and to combine
these results. Finally, in 
Section \ref{sec:add}, we present three additional identities for the
partial Bell polynomials. 
Throughout, we write $B_{n,k}$ as a shorthand for 
$B_{n,k}(x_1, \ldots, x_{n-k+1})$, unless explicit reference to the indeterminates is
critical. 

\section{Preliminaries on weighted integer compositions}
An integer composition of a nonnegative integer $n$ is a $k$-tuple $(\pi_1,\ldots,\pi_k)$,
for $k\ge 1$, 
of nonnegative integers such that
$\pi_1+\cdots+\pi_k=n$. We call $k$ the \emph{number of parts}. Note
that order of part matters, and this distinguishes integer
compositions from \emph{integer partitions}. Now, to generalize, we
may consider 
$f$-\emph{colored} integer compositions, where
each possible \emph{part size} $s\in\mathbb{N}=\set{0,1,2,\ldots}$ may
come in $f(s)$ different colors, whereby $f:\mathbb{N}\rightarrow\mathbb{N}$.
For example, the integer $n=4$ has 
nine distinct $f$-colored integer compositions, for $f(0)=2$,
$f(1)=f(2)=1$, $f(3)=f(4)=f(5)=\cdots=0$, with $k=3$ parts, namely,
\begin{align*}
  (2,2,0),(2,2,0^*),(2,1,1),(1,2,1),(2,0,2),(2,0^*,2),(1,1,2),(0,2,2),(0^*,2,2), 
\end{align*}
where we use a star superscript to distinguish the two different
colors 
of part size $0$. For \emph{weighted integer compositions} we let, more
generally, $f(s)$ be an arbitrary real number (or even a value in a
commutative ring), the weight of part size $s$. Colored integer
compositions have been discussed in \cite{guo,mansour,shapcott}, and
weighted integer 
compositions have been under review in Eger \cite{eger}, but have been
investigated as early as Hoggatt and Lind's \cite{hoggatt} work. 

Let $\binom{k}{n}_f$ denote the \emph{number} of $f$-weighted integer
compositions of $n$ with $k$ parts, when the range of $f$ is
$\mathbb{N}$, and let $\binom{k}{n}_f$ denote the \emph{total 
weight} 
of all
$f$-weighted integer compositions of $n$ with $k$ parts, when the
range of $f$ is the set of 
reals $\mathbb{R}$. 
We then have the following theorem. 
\begin{theorem}\label{th:1}
Let $k,n\ge 0$ be integers, and let $f:\nn\rightarrow\real$ be
arbitrary. Then the following identities hold. 
\begin{align}
\label{eq:1}
\binom{k}{n}_f &= \sum f(\pi_1)\cdots f(\pi_k),\\
\label{eq:2}
\binom{k}{n}_f &= \sum
\binom{k}{\ell_0,\ell_1,\ldots,\ell_n}f(0)^{\ell_0}\cdots f(n)^{\ell_n},
\end{align}
where the sum in \eqref{eq:1} is over all solutions
in nonnegative
integers $\pi_1,\ldots,\pi_k$ of $\pi_1+\cdots+\pi_k=n$, and the sum
in \eqref{eq:2} is 
over all solutions in nonnegative integers $\ell_0,\ldots,\ell_n$ of
$\ell_0+\cdots+\ell_n=k$ and
$0\ell_0+1\ell_1+\cdots+n\ell_n=n$. Finally,
$\binom{k}{\ell_0,\ell_1,\ldots,\ell_n}=\frac{k!}{\ell_0!\ell_1!\cdots\ell_n!}$
denote the multinomial 
coefficients. 
\end{theorem}
The proof of Theorem \ref{th:1} is simple. Identity \eqref{eq:1} is a
direct application of the definition of $\binom{k}{n}_f$ and
\eqref{eq:2} follows, combinatorially, from \eqref{eq:1} by rewriting
the summation over integer compositions into a summation over integer
partitions and then adjusting each term in the sum appropriately (in
particular, the multinomial coefficients account for distributing the
parts in partitions). Alternatively, when $\binom{k}{n}_f$ is defined
as coefficient of a certain polynomial, then \eqref{eq:2} can also be
arrived at via application of the multinomial theorem (see
\cite{eger}) or the formula of Fa\`{a} di Bruno 
\cite{johnson} for
the higher order derivatives of composite functions (see \cite{hoggatt}). 

Interpreting $f(s)$, for $s\in\nn$, as indeterminates, Theorem
\ref{th:1} identity \eqref{eq:2} immediately implies that
$f$-weighted integer compositions 
`generalize' partial Bell polynomials.
\begin{corollary}\label{corr:1}
  Let $k,n\ge 0$ be integers. Then: 
  \begin{align}
    \frac{k!}{n!}B_{n,k}(x_1,x_2,\ldots,x_{n-k+1}) = \binom{k}{n}_f,
  \end{align}
  whereby
  \begin{align*}
    & f(0)=f(n-k+2)=f(n-k+3)=\cdots=0,\\ 
    & f(s) = \frac{x_s}{s!},
    \text{ for } s\in\set{1,2,\ldots,n-k+1}.
  \end{align*}
\end{corollary}
Corollary \ref{corr:1} has been established in \cite{belbachir} in the
more particular setting 
of restricted integer compositions, or, classical `multinomial/extended
binomial coefficients', where, in our notation, $f(s)=1$
for all $s\in\set{1,\ldots,q}$, 
for some positive integer $q$.
This 
yielded the
conclusion that partial Bell polynomials `generalize' 
restricted integer 
compositions insofar as
$B_{n,k}(1!,2!,\ldots,q!,0,\ldots)=\frac{n!}{k!}\binom{k}{n}_{f}$, where
$f$ is the indicator function on $\set{1,\ldots,q}$. 

Representation \eqref{eq:1} in Theorem \ref{th:1} is very useful on
its own since it 
captures 
an equivalence between 
$f$-weighted integer
compositions 
and
distributions of sums of independent and identically distributed
(i.i.d.) discrete random
variables as explicated in the following lemma. 
\begin{lemma}\label{lemma:1}
  Let $k,n\ge 0$ be integers and let $f:\nn\goesto\real$ be
  arbitrary. Then there exist i.i.d.\ nonnegative integer-valued random variables
  $X_1,\ldots,X_k$ such that $\binom{k}{n}_f$ is given as the distribution of the
  sum of 
  $X_1,\ldots,X_k$ (times a suitable normalization factor).

  Conversely, the distribution $P_f[X_1+\cdots+X_k=n]$ of arbitrary
  i.i.d.\ 
  nonnegative integer-valued random variables $X_1,\ldots,X_k$ with
  common distribution function $f$ is given by $\binom{k}{n}_f$. 
\end{lemma}
\begin{proof}
  Consider $\binom{k}{n}_f$. When $f$ is zero almost everywhere, i.e.,
  $f(s)=0$ for all $s>x$, for some $x\in\nn$, then let $\bar{F}$
  denote the sum  $\sum_{s'\in\nn} f(s')$ and let
  $g(s)=\frac{f(s)}{\bar{F}}$ for all $s\in\nn$. Consider the i.i.d.\ random
  variables $X_1,\ldots,X_k$ with common distribution function $g$,
  i.e., $P[X_i=s]=g(s)$. 
  By definition, the distribution of the sum $X_1+\cdots+X_k$ is given
  as
  \begin{align*}
    P_g[X_1+\cdots+X_k=n] &= \sum_{\pi_1+\cdots+\pi_k=n}
    P[X_1=\pi_1]\cdots P[X_k=\pi_k]\\ 
    & = \sum_{\pi_1+\cdots+\pi_k=n}
    g(\pi_1)\cdots g(\pi_k) \\&=
    \frac{1}{\bar{F}^n}\sum_{\pi_1+\cdots+\pi_k=n} f(\pi_1)\cdots
    f(\pi_k) = \frac{1}{\bar{F}^n}\binom{k}{n}_f.
  \end{align*}
  When $f$ is not zero almost everywhere, then note that
  \begin{align*}
    \binom{k}{n}_f = \binom{k}{n}_{\hat{f}} =
    {\bar{\hat{F}}^n}P_{\hat{g}}[Y_1+\cdots+Y_k=n],
  \end{align*}
  whereby $\hat{f}(s)=f(s)$ for $s\le n$ and
$\hat{f}(s)=0$ for $s>n$, and $Y_1,\ldots,Y_k$ are i.i.d.\ random
  variables with common distribution function
  $\hat{g}(s)=\frac{\hat{f}(s)}{\bar{\hat{F}}}$, where
  $\bar{\hat{F}}=\sum_{s'\in\nn} \hat{f}(s')$.  

  Conversely, the distribution of the sum of i.i.d.\ nonnegative
  integer-valued random variables $X_1,\ldots,X_k$ with common
  distribution function $f$ is given by, as
  above, 
  \begin{align*}
    P_f[X_1+\cdots+X_k=n] &= \sum_{\pi_1+\cdots+\pi_k=n}
    P[X_1=\pi_1]\cdots P[X_k=\pi_k] \\
    &= \sum_{\pi_1+\cdots+\pi_k=n}
    f(\pi_1)\cdots f(\pi_k) 
    =\binom{k}{n}_f,
  \end{align*}
  applying Theorem \ref{th:1} in the last equality. 
\end{proof}
Lemma \ref{lemma:1} has appeared in \cite{eger} and, in the special case 
when $f$ is the discrete uniform measure, in
\cite{belbachir}.  Lemma \ref{lemma:1} allows us to prove properties
of the 
weighted integer compositions $\binom{k}{n}_f$ by referring to
properties of the distribution of the sum of discrete random
variables, which is 
oftentimes convenient, as we shall see below. 

\section{Identities and proofs}
The three identities for partial Bell polynomials that Cvijovi\'c
\cite{cvijovic} introduces are the following: 
\begin{align}
  \label{eq:id1}
  B_{n,k} &=
  \frac{1}{x_1}\frac{1}{n-k}\sum_{\alpha=1}^{n-k}\binom{n}{\alpha}\left[(k+1)-\frac{n+1}{\alpha+1}\right]x_{\alpha+1}B_{n-\alpha,k},
  \\
  \label{eq:id2}
  B_{n,k_1+k_2} &=
  \frac{k_1!k_2!}{(k_1+k_2)!}\sum_{\alpha=0}^n\binom{n}{\alpha}B_{\alpha,k_1}B_{n-\alpha,k_2},\\
  \label{eq:id3}
  B_{n,k+1} &=
  \frac{1}{(k+1)!}\sum_{\alpha_1=k}^{n-1}\cdots\sum_{\alpha_k=1}^{\alpha_{k-1}-1}
  \binom{n}{\alpha_1}\cdots\binom{\alpha_{k-1}}{\alpha_k}x_{n-\alpha_1}\cdots x_{\alpha_k-1-\alpha_k}x_{\alpha_k}.
\end{align}
As shown in Cvijovi\'c \cite{cvijovic}, \eqref{eq:id3} easily follows
inductively from 
\eqref{eq:id2}. Therefore, we concentrate on the identities \eqref{eq:id1}
and \eqref{eq:id2}. 
Throughout, we let $k,n$ be nonnegative integers and $f$ be a function
$f:\nn\goesto\real$. As we 
will see now, identity \eqref{eq:id2} 
may be seen as a special case of the convolution formula for the sum of
discrete random 
variables. 

\begin{lemma}\label{lemma:convolution}
  Let $n$, $k_1$ and $k_2$ be nonnegative integers. Then: 
  \begin{align*}
  \binom{k_1+k_2}{n}_f = \sum_{x+y=n}\binom{k_1}{x}_f\binom{k_2}{y}_f.
  \end{align*}
\end{lemma}
\begin{proof}
  By our previous discussion, it suffices to prove the lemma for sums
  of i.i.d.\ integer-valued random variables
  $X_1,\ldots,X_{k_1},Y_1,\ldots,Y_{k_2}$. Now, 
  \begin{align*}
    P_f[(X_1+\cdots+X_{k_1})+(Y_1+\cdots+Y_{k_2})=n] & \\ =\sum_{x+y=n}
    P_f[X_1+\cdots+&X_{k_1}=x]P_f[Y_{1}+\cdots+Y_{k_2}=y] 
  \end{align*}
  by the discrete convolution formula for discrete random variables. 

  To formally complete the proof,\footnote{We omit this
    straightforward step in all
    subsequent proofs.} we apply Lemma \ref{lemma:1} (or,
  more precisely, its proof), leading to: 
  \begin{align*}
    \binom{k_1+k_2}{n}_f &= \binom{k_1+k_2}{n}_{\hat{f}} =
    \bar{\hat{F}}^nP_{\hat{g}}[X_1+\cdots+X_{k_1+k_2}=n]\\
    &=\bar{\hat{F}}^n\sum_{x+y=n}P_{\hat{g}}[X_1+\cdots+X_{k_1}=x]P_{\hat{g}}[Y_{1}+\cdots+Y_{k_2}=y]
    \\
    &=
    \sum_{x+y=n}\bar{\hat{F}}^xP_{\hat{g}}[X_1+\cdots+X_{k_1}=x]\bar{\hat{F}}^yP_{\hat{g}}[Y_{1}+\cdots+Y_{k_2}=y]\\
    &= \sum_{x+y=n} \binom{k_1}{x}_{\hat{f}}\binom{k_2}{y}_{\hat{f}} =
    \sum_{x+y=n}\binom{k_1}{x}_{f}\binom{k_2}{y}_{f}.  
  \end{align*}
\end{proof}
\begin{proof}[Proof of identity \eqref{eq:id2}]
  Using Lemma \ref{lemma:convolution} and Corollary \ref{corr:1}, we
  obtain that 
  \begin{align*}
    B_{n,k_1+k_2} &= \frac{n!}{(k_1+k_2)!}\binom{k_1+k_2}{n}_f =
    \frac{n!}{(k_1+k_2)!}\sum_{x+y=n}\binom{k_1}{x}_f\binom{k_2}{y}_f
    \\
    &= 
    \frac{n!}{(k_1+k_2)!}\sum_{x+y=n}\frac{k_1!}{x!}B_{x,k_1}\frac{k_2!}{y!}B_{y,k_2} 
    \\ &= 
    \frac{k_1!k_2!}{(k_1+k_2)!}\sum_{x+y=n}\frac{n!}{x!y!}B_{x,k_1}B_{y,k_2}
    \\ &= \frac{k_1!k_2!}{(k_1+k_2)!}\sum_{x+y=n}\binom{n}{x}B_{x,k_1}B_{y,k_2}.
  \end{align*}
\end{proof}
The proof of the following identity for weighted integer compositions
which directly entails identity 
\eqref{eq:id1} 
is a straightforward 
variation 
of the proof outlined in DePril \cite{depril}, who considers sums of
integer-valued random variables. 
\begin{lemma}\label{lemma:depril}
  Let $k,n\ge 0$ be integers and 
  let $f:\nn\goesto\real$ be arbitrary with $f(0)=0$ and $f(1)\neq 0$. Then: 
  \begin{align*}
    \binom{k}{n}_{f} = \frac{1}{f(1)(n-k)}\sum_{s\ge
      1}\left(k+1-\frac{n+1}{s+1}\right)(s+1)f(s+1)\binom{k}{n-s}_f.
  \end{align*}
\end{lemma}
\begin{proof}
  Let $X_1,\ldots,X_{k+1}$ be i.i.d.\ nonnegative integer-valued random
  variables, with distribution function $f$. Consider the conditional
  expectation $\Exp[X_1\sd X_1+\ldots+X_{k+1}=n+1]$. Due to
  identical distribution of the variables and linearity of
  $\Exp[\cdot\sd \cdot]$, it follows that $\Exp[X_1\sd 
    X_1+\ldots+X_{k+1}=n+1] = 
  \frac{\Exp[X_1+\ldots+X_{k+1}\sd X_1+\ldots+X_{k+1}=n+1]}{k+1} =
  \frac{n+1}{k+1}$. Thus, $\Exp[\frac{k+1}{n+1}X_1-1\sd
    X_1+\ldots+X_{k+1}=n+1]=0$, which means that 
  \begin{align*}
    0& = \sum_{s=1}^n
    (\frac{k+1}{n+1}s-1)\frac{P_f[X_1=s,X_1+\cdots+X_{k+1}=n+1]}{P_f[X_1+\ldots+X_{k+1}=n+1]}
    \\
    & = \frac{1}{\binom{k+1}{n+1}_f}\sum_{s=1}^n
      (\frac{k+1}{n+1}s-1){P[X_1=s]\cdot P_f[X_2+\cdots+X_{k+1}=n+1-s]}\\
    &= \frac{1}{\binom{k+1}{n+1}_f}\sum_{s=1}^n(\frac{k+1}{n+1}s-1)f(s)\binom{k}{n+1-s}_f,
  \end{align*}
  so that rewriting and shifting indices lead to the required expression.
\end{proof}
We omit the proof of identity \eqref{eq:id1} since it is a simple
application of Lemma \ref{lemma:depril}.
\section{More identities for the partial Bell polynomials}\label{sec:add}
We mention the following three additional identities for partial Bell
polynomials,
\begin{align}
  \label{eq:4}
  B_{n,k} &= \frac{1}{k}\sum_{\alpha\ge
    0}\binom{n}{\alpha}x_{\alpha}B_{n-\alpha,k-1},\\
  \label{eq:5}
  B_{n,k} &= \sum_{\alpha\ge
    1}\binom{n-1}{\alpha-1}x_{\alpha}B_{n-\alpha,k-1},\\
  \label{eq:6}
  B_{n,k}(x_1,\ldots,x_{n-k+1}) &= \sum_{\alpha\ge 0}
  \binom{n}{\alpha}x_1^\alpha
  B_{n-\alpha,k-\alpha}(0,x_2,\ldots,x_{n-k+1}).
\end{align}
(note that \eqref{eq:4} is a special case of \eqref{eq:id2}) which
straightforwardly follow from the following corresponding identities
for weighted 
integer compositions, 
\begin{align}
  \label{eq:4_ana}
  \binom{k}{n}_f &= \sum_{s\ge 0} f(s)\binom{k-1}{n-s}_f,\\
  \label{eq:5_ana}
  \binom{k}{n}_f &= \frac{k}{n}\sum_{s\ge 1} sf(s)\binom{k-1}{n-s}_f,\\
  \label{eq:6_ana}
  \binom{k}{n}_f &= \sum_{i\ge 0} f(r)^i\binom{k}{i}\binom{k-i}{n-ri}_{\tilde{f}}.
\end{align}
In Equation \eqref{eq:6_ana}, $\tilde{f}(r)=0$ and $\tilde{f}(s)=f(s)$
for all $s\neq r$ (note 
that \eqref{eq:6} is a special case of \eqref{eq:6_ana} in which
$r=1$). Proofs of identities \eqref{eq:4_ana} to \eqref{eq:6_ana} can,
e.g., be found in 
Fahssi \cite{fahssi} and Eger \cite{eger}, who also give further identities for
weighted integer compositions.

\end{document}